\documentclass{amsart}
\usepackage{latexsym, amssymb, amsmath}

\newtheorem{conj}{Conjecture}
\newtheorem{thm}{Theorem}[section]
\newtheorem{lemm}[thm]{Lemma}

\theoremstyle{remark}

\theoremstyle{definition}

\title{Biharmonic properly immersed submanifolds \\ in Euclidean spaces}

\author{Kazuo Akutagawa${}^{\ast}$} 
\author{Shun Maeta${}^{\dagger}$}
\thanks{
${}^{\ast}$
supported in part by Grants-in-Aid for Scientific Research (B), 
Japan Society for the Promotion of Science, No.~24340008. \\
\qquad
${}^{\dagger}$
supported in part by 
Research Fellowships of the Japan Society for the Promotion of Science for Young Scientists, No.~23-6949.
}
\thanks{2010~{\em Mathematics Subject Classification.}~primary 58E20, secondary 53C43, 53A07}
\date{June, 2011.} 

\begin{document} 
\maketitle 
\markboth{Biharmonic properly immersed submanifolds in Euclidean spaces} 
{Kazuo Akutagawa and Shun Maeta}

\begin{abstract}
We consider a {\it complete} biharmonic immersed submanifold $M$ 
in a Euclidean space $\mathbb{E}^N$. 
Assume that the immersion is {\it proper}, that is, 
the preimage of every compact set in $\mathbb{E}^N$ is also compact in $M$. 
Then, we prove that $M$ is minimal. 
It is considered as an affirmative answer to 
the global version of Chen's conjecture for biharmonic submanifolds. 
%\keywords{Biharmonic submanifold \and Chen's conjecture}
\end{abstract}

\section{Introduction}
\label{intro}

Let $M$ be an $n$-dimensional connected immersed submanifold 
in the Euclidean $N$-space $\mathbb{E}^N\ (n < N)$ 
and ${\bf x}$ its position vector field. 
Then, it is well known that 
\begin{equation}\label{mean curv} 
\Delta {\bf x} = n {\bf H}, 
\end{equation}
where $\Delta$ and ${\bf H}$ denote respectively the (non-positive) Laplace operator 
and the mean curvature vector field of $M$. 
The above equation shows particularly that 
$M$ is minimal, that is, ${\bf H} = 0$ if and only if 
the isometric immersion ${\bf x} : (M, g) \rightarrow \mathbb{E}^N$ 
is a harmonic map. 
Here, $g$ denotes the induced Riemannian metric on $M$ from ${\bf x}$. 
$M$ is said to be {\it biharmonic} if ${\bf H}$ satisfies the following: 
\begin{equation}\label{biharmonic}
\Delta {\bf H} = \frac{1}{n} \Delta^2 {\bf x} = 0. 
\end{equation} 
It is obvious that every minimal submanifold is biharmonic. 
We also note that $M$ is biharmonic if and only if 
${\bf x}$ is a biharmonic map. 

For biharmonic submanifolds, 
there is an interesting problem, namely, Chen's Conjecture 
(cf.~\cite{Chen}): 

\begin{conj} 
Any biharmonic submanifold $M$ in $\mathbb{E}^N$ is minimal. 
\end{conj} 

There are many affirmative partial answers to Conjecture~$1$ 
(cf.~\cite{Chen}, \cite{Chen-Ishikawa-1}, \cite{Chen-Ishikawa-2}, \cite{Dimi}, \cite{Hasanis-Vlachos}, \cite{Leuven}). 
In particular, there are some complete affirmative answers 
if $M$ is one of the following: 
(a) a curve \cite{Dimi}, 
(b) a surface in $\mathbb{E}^3$ \cite{Chen}, 
(c) a hypersurface in $\mathbb{E}^4$ \cite{Hasanis-Vlachos}, \cite{Leuven}. 

On the other hand, 
since there is no assumption of {\it completeness} for submanifolds in Conjecture~1, 
in a sense it is a problem in {\it local} differential geometry.  
In this article, we reformulate Conjecture~1 into a problem 
in {\it global} differential geometry as the following (cf.~\cite{N-U-1}, \cite{N-U-2}):  

\begin{conj} 
Any {\rm complete} biharmonic immersed submanifold in $\mathbb{E}^N$ is minimal. 
\end{conj} 

An immersed submanifold $M$ in $\mathbb{E}^N$ is said to be {\it properly immersed} 
if the immersion $M \rightarrow \mathbb{E}^N$ is a proper map. 
Here, we remark that the properness of the immersion implies the completeness of $(M, g)$.  
Our main result is the following, which gives an affirmative partial answer to Conjecture~2:  

\begin{thm}\label{Main Result} 
Any biharmonic {\rm properly immersed} submanifold $M$ in $\mathbb{E}^N$ is minimal. 
\end{thm} 

For proving Theorem~\ref{Main Result}, 
the basic tool is the generalized maximum principle technique
developed in Cheng-Yau~\cite{Cheng-Yau} as follows: 

{\it Let $(M, g)$ be a complete manifold
whose Ricci curvature ${\rm Ric}_g$ is bounded from below. 
Let $u$ be a smooth nonnegative function on $M$. 
Assume that there exists a positive constant $k > 0$ such that
\begin{equation}\label{GMP}
\Delta u \geq k u^2\quad {\rm on}\ \ M.
\end{equation} 
Then, $u = 0$ on $M$.} 

The outline of proof of the generalized maximum principle is the following. 
For a fixed point $x_0 \in M$ and each large positive constant $\rho > 0$, 
consider the following smooth function 
$$ 
f(x) := (\rho^2 - r(x)^2)^2 u(x)\quad {\rm for}\ \ x \in \overline{B_{\rho}(x_0)}, 
$$ 
where $r(x) := {\rm dist}_g(x, x_0)$ 
and $\overline{B_{\rho}(x_0)} := \{ x \in M\ |\ r(x) \leq \rho \}$ 
denote respectively the distance from $x_0$ 
and the closed geodesic ball of radius $\rho$ centered at $x_0$. 
Then, the inequality (\ref{GMP}) implies that 
$$ 
f(p) \leq c \rho^3\quad {\rm at\ a\ maximum\ point}\ \ p \in 
B_{\rho}(x_0) := \{ x \in M\ |\ r(x) < \rho \}, 
$$ 
and hence 
\begin{equation}\label{crucial esti}
u(x) \leq \frac{c \rho^3}{(\rho^2 - r(x)^2)^2}\quad {\rm for}\ \ x \in B_{\rho}(x_0). 
\end{equation} 
Letting $\rho \nearrow \infty$ in the above inequality, 
we then get that $u = 0$ on $M$. 
Here, $c > 0$ is a positive constant depending only on $k$, dim~$M$ and 
the constant $\kappa \geq 0$ satisfying ${\rm Ric}_g \geq - \kappa$ on $M$. 
The assumption of Ricci curvature bound from below 
is necessary for the estimate of $\Delta r(p)$ from above 
(see \cite{Yau} for details). 

When $(M, g)$ is a Riemannian immersed submanifold in $\mathbb{E}^N$, 
it is impossible to get such Ricci curvature bound from below 
without an assumption of boundedness for the second fundamental form $h$ of $M$. 
However, for Conjecture~2, any assumption for $h$ is artificial in some sense. 
To overcome this difficulty, we consider the function 
$$ 
F(x) := (\rho^2 - |{\bf x}(x)|^2)^2 u(x)\quad {\rm for}\ \ 
x \in M \cap {\bf x}^{-1}\big{(} \overline{{\bf B}_{\rho}} \big{)}  
$$ 
instead of $f(x)$, where $|{\bf x}(x)|^2 := \langle{\bf x}(x), {\bf x}(x)\rangle$ 
denotes the square-norm of the position vector 
${\bf x}(x)$ of $x \in M$ in $\mathbb{E}^N$ 
and $\overline{{\bf B}_{\rho}} := \{ {\bf x} \in \mathbb{E}^N~|~|{\bf x}| \leq \rho \}$.
From the formula (\ref{mean curv}), 
we then get 
$$  
|\Delta {\bf x}(x)| = n |{\bf H}(x)|. 
$$ 
Moreover if $M$ is biharmonic, 
by the harmonicity (\ref{biharmonic}) combined with the above estimate, 
one can obtain a similar estimate to (\ref{crucial esti}) 
for $u(x) := |{\bf H}(x)|^2$ especially (see Section~$3$ for details). 

The remaining sections are organized as follows. 
Section~$2$ contains some necessary definitions and preliminary geometric results. 
Section~$3$ is devoted to the proof of Theorem~\ref{Main Result}.

\noindent 
{\bf Acknowledgements.} 
The first author would like to thank 
Reiko Aiyama, Nobumitsu Nakauchi and Hajime Urakawa for helpful discussions. 
He also would like to thank Luis Al\'ias and Reiko Miyaoka for useful comments.

\section{\bf Preliminaries}\label{Pre} 

Let $M$ be an $n$-dimensional immersed submanifold in $\mathbb{E}^N$, 
${\bf x} : M \rightarrow \mathbb{E}^N$ its immersion  
and $g$ its induced Riemannian metric. 
For simplicity, we often identify $M$ with its immersed image ${\bf x}(M)$ in every local arguments. 
Let $\nabla$ and $D$ denote respectively the Levi-Civita connections 
of $(M, g)$ and $\mathbb{E}^N = (\mathbb{R}^N, \langle\ ,\ \rangle)$. 
For any vector fields $X, Y \in \frak{X}(M)$, 
the Gauss formula is given by 
$$ 
D_XY = \nabla_XY + h(X, Y), 
$$ 
where $h$ stands for the second fundamental form of $M$ in $\mathbb{E}^N$. 
For any normal vector field $\xi$, the Weingarten map $A_{\xi}$ with respect to $\xi$ 
is given by 
$$ 
D_X\xi = - A_{\xi}X + \nabla^{\perp}_X\xi, 
$$ 
where $\nabla^{\bot}$ stands for the normal connection of the normal bundle of $M$ in $\mathbb{E}^N$. 
It is well known that $h$ and $A$ are related by 
$$ 
\langle h(X, Y), \xi \rangle = \langle A_{\xi}X, Y \rangle. 
$$ 

For any $x \in M$, 
let $\{e_1, \cdots, e_n, e_{n+1}, \cdots, e_N\}$ be an orthonormal basis of $\mathbb{E}^N$ at $x$ 
such that $\{e_1, \cdots, e_n\}$ is an orthonormal basis of $T_xM$. 
Then, $h$ is decomposed as at $x$ 
$$ 
h(X, Y) = \Sigma_{\alpha=n+1}^N h_{\alpha}(X, Y)e_{\alpha}. 
$$ 
The mean curvature vector ${\bf H}$ of $M$ at $x$ is also given by 
$$ 
{\bf H}(x) = \frac{1}{n} \Sigma_{i = 1}^n h(e_i, e_i) = \Sigma_{\alpha=n+1}^N H_{\alpha}(x)e_{\alpha},\qquad 
H_{\alpha}(x) := \frac{1}{n} \Sigma_{i = 1}^n h_{\alpha}(e_i, e_i).  
$$ 
It is well know that the necessary and sufficient conditions for $M$ in $\mathbb{E}^N$ 
to be biharmonic, namely $\Delta {\bf H} = 0$, 
are the following (cf.~\cite{Chen}, \cite{Chen-Ishikawa-1}, \cite{Chen-Ishikawa-2}): 
\begin{equation}\label{N-S} 
\begin{cases} 
\ \ \Delta^{\perp} {\bf H} - \Sigma_{i=1}^n h(A_{\bf H}e_i, e_i) = 0, \\ 
\ \ n~\nabla |{\bf H}|^2 + 4~{\rm trace}~A_{\nabla^{\perp} {\bf H}} = 0, \\  
\end{cases} 
\end{equation} 
where $\Delta^{\perp}$ is the (non-positive) Laplace operator associated with the normal connection $\nabla^{\perp}$. 

From the first equation of (\ref{N-S}), we have the following. 

\begin{lemm} 
Let $M = (M, g)$ be a biharmonic immersed submanifold in $\mathbb{E}^N$. 
Then, the following inequality for $|{\bf H}|^2$ holds 
\begin{equation}\label{key} 
\Delta |{\bf H}|^2 \geq 2 n |{\bf H}|^4. 
\end{equation} 
\end{lemm} 

\begin{proof} 
Under the above notations, 
the first equation of (\ref{N-S}) implies that, at each $x \in M$, 
\begin{align}\label{ell-ineq} 
\Delta |{\bf H}|^2 
& = 2~\Sigma_{i=1}^n \langle \nabla_{e_i}^{\perp} {\bf H}, \nabla_{e_i}^{\perp} {\bf H} \rangle 
+ 2~\langle \Delta^{\perp} {\bf H}, {\bf H} \rangle \notag \\ 
& \geq 2~\Sigma_{i=1}^n \langle h(A_{\bf H} e_i, e_i), {\bf H} \rangle \\ 
& = 2~\Sigma_{i=1}^n \langle A_{\bf H} e_i, A_{\bf H} e_i \rangle .\notag   
\end{align} 
When ${\bf H}(x) \ne 0$, set $e_N := \frac{{\bf H}(x)}{|{\bf H}(x)|}$. 
Then, ${\bf H}(x) = H_N(x) e_N$ and $|{\bf H}(x)|^2 = H_N(x)^2$. 
From (\ref{ell-ineq}), we have at $x$ 
\begin{align}\label{ell-ineq} 
\Delta |{\bf H}|^2 
& \geq 2~H_N^2~\Sigma_{i=1}^n \langle A_{e_N} e_i, A_{e_N} e_i \rangle\notag \\ 
& = 2~|{\bf H}|^2~|h_N|_g^2 \notag\\ 
& \geq 2 n~|{\bf H}|^2~H_N^2 \notag\\ 
& = 2 n~|{\bf H}|^4 \notag. 
\end{align} 
Even when ${\bf H}(x) = 0$, the above inequality~(\ref{key}) still holds at $x$. 
This completes the proof. 
\end{proof}

\section{\bf Proof of Main Theorem}\label{EC} 

\begin{proof}[Proof of Theorem~\ref{Main Result}]\quad 
If $M$ is compact, applying the standard maximum principle to the elliptic inequality (\ref{key}), 
we have that ${\bf H} = 0$ on $M$. 
But, there is no compact minimal submanifold in $\mathbb{E}^N$.
Hence, this case never occur. 
Therefore, we may assume that $M$ is noncompact. 
Suppose that ${\bf H}(x_0) \ne 0$ at some point $x_0 \in M$. 
Then, we will lead a contradiction. 

Set 
$$ 
u(x) := |{\bf H}(x)|^2\quad {\rm for}\ \ x \in M. 
$$  
For each $\rho > 0$, consider the function 
$$ 
F(x) = F_{\rho}(x) := (\rho^2 - |{\bf x}(x)|^2)^2 u(x)\quad {\rm for}\ \ 
x \in M \cap {\bf x}^{-1}\big{(} \overline{{\bf B}_{\rho}} \big{)}.  
$$ 
Then, there exists $\rho_0 > 0$ such that $x_0 \in {\bf x}^{-1}\big{(} {\bf B}_{\rho_0} \big{)}$. 
For each $\rho \geq \rho_0$, 
$F = F_{\rho}$ is a nonnegative function which is not identically zero 
on $M \cap {\bf x}^{-1}\big{(} \overline{{\bf B}_{\rho}} \big{)}$. 
Take any $\rho \geq \rho_0$ and fix it. 
Since $M$ is properly immersed in $\mathbb{E}^N$, $M \cap {\bf x}^{-1}\big{(} \overline{{\bf B}_{\rho}} \big{)}$ is compact. 
By this fact combined with $F = 0$ on $M \cap {\bf x}^{-1} \big{(} \partial\overline{{\bf B}_{\rho}} \big{)}$, 
there exists a maximum point $p \in M \cap {\bf x}^{-1}\big{(} {\bf B}_{\rho} \big{)}$ of $F = F_{\rho}$ such that $F(p) > 0$. 
We have $\nabla F = 0$ at $p$, and hence 
\begin{equation}\label{grad}
\frac{\nabla u}{u} = \frac{2~\nabla |{\bf x}(x)|^2}{\rho^2 - |{\bf x}(x)|^2}\quad {\rm at}\ \ p. 
\end{equation} 
We also have that $\Delta F \leq 0$ at $p$. 
Combining this with (\ref{grad}), we obtain 
\begin{equation}\label{Lap} 
\frac{\Delta u}{u} \leq \frac{6~|\nabla |{\bf x}(x)|^2|_g^2}{(\rho^2 - |{\bf x}(x)|^2)^2} 
+ \frac{2~\Delta |{\bf x}(x)|^2}{\rho^2 - |{\bf x}(x)|^2}\quad {\rm at}\ \ p.   
\end{equation}   
From (\ref{biharmonic}), we note 
\begin{equation}\label{sub} 
\begin{cases} 
\ \ \Delta |{\bf x}(x)|^2 = 2~\Sigma_{i=1}^n |\nabla_{e_i} {\bf x}(x)|^2 + 2~\langle \Delta {\bf x}(x), {\bf x}(x) \rangle 
\leq 2 n + 2 n |{\bf H}|\cdot |{\bf x}(x)|, \\ 
\ \ |\nabla |{\bf x}(x)|^2|_g^2 \leq 4 n |{\bf x}(x)|^2.   
\end{cases} 
\end{equation} 
It then follows from (\ref{key}), (\ref{Lap}) and (\ref{sub}) that 
$$ 
u(p) \leq \frac{12  |{\bf x}(p)|^2}{(\rho^2 - |{\bf x}(p)|^2)^2} 
+ \frac{2  (1 + \sqrt{u(p)} |{\bf x}(p)|)}{\rho^2 - |{\bf x}(p)|^2},  
$$ 
and hence 
$$ 
F(p) \leq 12  |{\bf x}(p)|^2 + 2  (\rho^2 - |{\bf x}(p)|^2) + 2  \sqrt{F(p)} |{\bf x}(p)|.  
$$ 
Therefore, there exists a positive constant $c > 0$ such that 
$$ 
F(p) \leq c \rho^2. 
$$ 

Since $F(p)$ is the maximum of $F = F_{\rho}$, we have 
$$ 
F(x) \leq F(p) \leq c \rho^2\quad {\rm for}\ \ x \in M \cap {\bf x}^{-1}\big{(} \overline{{\bf B}_{\rho}} \big{)}, 
$$ 
and hence 
\begin{equation}\label{Final} 
|{\bf H}(x)|^2 = u(x) \leq \frac{c \rho^2}{(\rho^2 - |{\bf x}(x)|^2)^2}\quad 
{\rm for}\ \ x \in M \cap {\bf x}^{-1}\big{(} {\bf B}_{\rho} \big{)}\quad {\rm and}\ \ \rho \geq \rho_0. 
\end{equation} 
Letting $\rho \nearrow \infty$ in (\ref{Final}) for $x = x_0$, 
we have that 
$$ 
|{\bf H}(x_0)|^2 = 0. 
$$  
This contradicts our assumption that ${\bf H}(x_0) \ne 0$.  
Therefore, $M$ is minimal. 
\end{proof}  

\quad \\ 
\quad\\

\bibliographystyle{amsbook}

\quad\\
\quad\\

\noindent
Kazuo~AKUTAGAWA\\
Division of Mathematics, GSIS, Tohoku University, 
Sendai 980-8579, Japan.\\
e-mail:~akutagawa@math.is.tohoku.ac.jp\\

\noindent
Shun~MAETA~\\
Division of Mathematics, GSIS, Tohoku University, 
Sendai 980-8579, Japan.\\
e-mail:~maeta@ims.is.tohoku.ac.jp\\

\end{document}